\documentclass{jonasart}

\usepackage[shortlabels]{enumitem}
\usepackage{xltabular}

\newcolumntype{C}[1]{>{\centering\arraybackslash}p{#1}}

\title{Skew braces with no proper left ideals}

\authors{Cindy (Sin Yi) Tsang}

\authorinfo{Ochanomizu University,  Japan}{tsang.sin.yi@ocha.ac.jp}

\abstract{A skew brace $A = (A,\cdot,\circ)$ is said to be \textit{left-simple} if $A\neq1$ and it has no left ideal other than $1$ and $A$. The purpose of this paper is to give a partial classification of the finite left-simple skew braces. A result of Stefanello and Trappeniers implies that finite left-simple skew braces correspond precisely to minimal Hopf--Galois structures on finite Galois extensions of fields.}



\acknowledgments{This work is supported by JSPS KAKENHI 24K16891. The author would like to thank the referees for their valuable suggestions.}

\VOLUME{1}
\ISSUE{1}
\NUMBER{5}
\DOI{https://doi.org/10.46298/jonas.17536}
\editinfo{February 18, 2026}{May 30, 2026}{Vicent Pérez-Calabuig}

\begin{document}
	\section{Introduction}

	A \textit{skew brace} is a~set $A = (A,\cdot,\circ)$ equipped with two group operations $\cdot$ and $\circ$ such that the so-called brace relation
	\[
		a\circ (b\cdot c) = (a\circ b)\cdot a^{-1}\cdot (a\circ c)
	\]
	holds for all $a,b,c\in A$. For any $a\in A$, we write $a^{-1}$ and $\overline{a}$, respectively, for the inverses of $a$ in the groups $(A,\cdot)$ and $(A,\circ)$. We easily check that $(A,\cdot)$ and $(A,\circ)$ share the same identity $1$. For each $a\in A$, it is straightforward to verify that the map
	\[
		\lambda_a: A \longrightarrow A;\quad \lambda_a (x) = a^{-1}\cdot (a\circ x)
	\]
	is an automorphism of $(A,\cdot)$, and it is also not hard to show that
	\begin{equation}\label{eqn:lambda map}
	    \lambda: (A,\circ)\longrightarrow \mathrm{Aut}(A,\cdot);
		\quad
		a \longmapsto \lambda_a
	\end{equation}
	is a~group homomorphism. A \textit{left ideal} of $A$ is a~subgroup $I$ of $(A,\cdot)$ that is invariant under the action of $\lambda$, namely $\lambda_a (I)\subseteq I$ for all $a\in A$. Note that $I$ is a~subgroup of $(A,\circ)$ in this case. We introduce a~new definition.

	\begin{definition}
		We shall say that a~skew brace $A = (A,\cdot,\circ)$ is \textit{left-simple} if $A\neq 1$ and it has no left ideal other than $1$ and $A$.
	\end{definition}

	Analogous to the Classification of Finite Simple Groups, it is natural to ask for a~classification of the finite left-simple skew braces. The purpose of this paper is to give a~partial classification.

	Let $A= (A,\cdot,\circ)$ be any finite left-simple skew brace. The \textit{additive group} of $A$, namely $(A,\cdot)$, must be characteristically simple because its characteristic subgroups are all left ideals of $A$. Hence, we have $(A,\cdot)\simeq T^n$ for some finite simple group $T$ and $n\in \mathbb{N}$. We are able to prove a~complete classification when $T$ is abelian and when $n=1$. We give a~few fairly restrictive conditions that $A$ must satisfy when $T$ is non-abelian and $n\geq 2$.

	\begin{theorem}\label{thm1}
	    Let $A=(A,\cdot,\circ)$ be a~finite skew brace such that $(A,\cdot)\simeq C_p^n$ for some prime $p$ and $n\in \mathbb{N}$. Then $A$ is left-simple if and only if $n=1$.
	\end{theorem}

	\begin{proof}
		We may regard the group homomorphism
		\[
			\lambda: (A,\circ) \longrightarrow \mathrm{Aut}(A,\cdot)\simeq \mathrm{GL}_n (\mathbb{F}_p )
		\]
		as a~linear representation of the finite $p$-group $(A,\circ)$ over the field $\mathbb{F}_p$ that has characteristic $p$. The left ideals of the skew brace $A$ then coincide with the $\lambda$-invariant subspaces of the vector space $(A,\cdot)\simeq \mathbb{F}_p^n$. This means that $A$ is left-simple if and only if $\lambda$ is irreducible. But it is well-known in modular representation theory (e.g. see~\cite[Proposition~6.2.1]{Webb}) that $\lambda$ is irreducible if and only if it is the trivial representation. This proves the theorem.
	\end{proof}

	The following corollary is immediate from Theorem~\ref{thm1}.

	\begin{corollary}\label{cor}
	    Up to isomorphism, the finite left-simple skew braces with an abelian additive group are exactly $\mathbb{F}_p =(\mathbb{F}_p,+,+)$, where $p$ is any prime and $+$ is the usual addition.
	\end{corollary}

	\begin{theorem}\label{thm2}
	    Let $A=(A,\cdot,\circ)$ be a~finite skew brace such that $(A,\cdot)\simeq T^n$ for some non-abelian simple group $T$ and $n\in \mathbb{N}$. It is well-known that
		\[
			\mathrm{Aut}(T^n ) = \mathrm{Aut}(T)\wr S_n = \mathrm{Aut}(T)^n\rtimes S_n,
		\]
		where $S_n$ denotes the symmetric group on $n$ letters.
		\begin{enumerate}[(a)]
			\item In the case $n = 1$, we have $A$ is left-simple if and only if $A$ is an almost trivial skew brace, namely, $a\circ b = b\cdot a$ for all $a,b\in A$.

			\item In the case $n \geq 2$, if $A$ is left-simple, then $(A,\circ)$ is isomorphic to some subgroup of $\mathrm{Aut}(T^n )$ whose projection onto $S_n$ is a~transitive subgroup, and $\mathrm{Im}(\lambda)$ intersects trivially with $\mathrm{Inn}(A,\cdot)$.
		\end{enumerate}
	\end{theorem}

	\begin{remark}
		The motivation of this paper came from the study of minimal Hopf--Galois structures on finite Galois extensions. In the language of Hopf--Galois theory, our results above yield a~partial classification of the minimal Hopf--Galois structures on finite Galois extensions (see Section~\ref{sec:minimal}).
	\end{remark}

	\begin{remark}
		A skew brace $A=(A,\cdot,\circ)$ is said to be \textit{simple} if $A\neq 1$ and it has no ideal other than $1$ and $A$. Recall that an \textit{ideal} of $A$ is a~left ideal $I$ of $A$ that is normal in both $(A,\cdot)$ and $(A,\circ)$. In this case, we can naturally define a~quotient skew brace structure on the set
		\[
			A/I = \{a\cdot I: a\in A\} = \{a\circ I:a\in A\}.
		\]
		Thus, ideals in a~skew brace are a~natural analog of normal subgroups in a~group, and simple skew braces are arguably a~closer analog of simple groups than left-simple skew braces. A classification of finite simple skew braces is desirable, but it seems to be out of reach at the moment. For example, in a~simple skew brace $A = (A,\cdot,\circ)$, it is possible that $(A,\cdot)$ is abelian~\cite{simple-brace}.
	\end{remark}

	\section{Preliminaries}

	Let $A = (A,\cdot,\circ)$ be a~skew brace. For each $a\in A$, the maps
	\begin{gather*}
		\lambda_a: A\longrightarrow A; \quad \lambda_a (x) = a^{-1}\cdot (a\circ x)\\
		\rho_a: A\longrightarrow A; \quad \rho_a (x) = (a\circ x) \cdot a^{-1}
	\end{gather*}
	are automorphisms of $(A,\cdot)$. Also, let
	\[
		\mathrm{conj}(a): A\longrightarrow A;
		\quad
		\mathrm{conj}(a)(x) = a \cdot x \cdot a^{-1}
	\]
	denote the inner automorphism of $(A,\cdot)$ induced by $a$, so that
	\begin{equation}\label{eqn:lambda}
	    \rho_a =\mathrm{conj}(a)\lambda_a, \quad
		\lambda_a = \mathrm{conj}(a^{-1})\rho_a, \quad
		\mathrm{conj}(a) = \rho_a\lambda_a^{-1}
	\end{equation}
	hold. It is well-known that
	\begin{gather*}
		\lambda: (A,\circ)\longrightarrow \mathrm{Aut}(A,\cdot); \quad a\longmapsto \lambda_a\\
		\rho: (A,\circ)\longrightarrow \mathrm{Aut}(A,\cdot); \quad a\longmapsto \rho_a
	\end{gather*}
	are group homomorphisms. Also let
	\[
		\mathrm{conj}: (A,\cdot) \longrightarrow \mathrm{Aut}(A,\cdot);
		\quad
		a~\longmapsto \mathrm{conj}(a)
	\]
	denote the natural group homomorphism.

	For the homomorphisms $\lambda$ and $\rho$, we consider the product
	\[
		\Gamma:= \mathrm{Im}(\lambda)\mathrm{Im}(\rho)
	\]
	of their images. The next proposition is essentially~\cite[Proposition 2.1]{Tsang} and it implies that $\Gamma$ is subgroup of $\mathrm{Aut}(A,\cdot)$ by the normality of $\mathrm{Inn}(A,\cdot)$.

	\begin{proposition}\label{prop:trifactor} We have the equalities
		\[
			\Gamma = \mathrm{Im}(\lambda)\mathrm{Inn}(A,\cdot )=\mathrm{Im}(\rho)\mathrm{Inn}(A,\cdot).
		\]
	\end{proposition}

	\begin{proof}
		For any $a,b\in A$, a~simple calculation using~\eqref{eqn:lambda} yields that
		\begin{align*}
			\lambda_a\cdot \rho_b
			& = \lambda_a\mathrm{conj}(b)\lambda_b =\lambda_{a}\lambda_b\cdot\mathrm{conj}(\lambda_{b}^{-1}(b)),\\
			\lambda_a\cdot\mathrm{conj}(b)
			& = \mathrm{conj}(a^{-1})\rho_a\mathrm{conj}(b)= \rho_a\cdot \mathrm{conj}(\rho_a^{-1}(a^{-1})b),\\
			\rho_a\cdot\mathrm{conj}(b)
			& = \mathrm{conj}(\rho_a (b)^{-1})^{-1}\rho_a = \lambda_{\rho_a (b)^{-1}}\cdot \rho_{\rho_a (b)^{-1}}^{-1}\rho_a.
		\end{align*}
		The equalities now follow.
	\end{proof}

	Using the homomorphisms $\lambda$ and $\rho$, there are two left ideals
	\[
		J_1:= \mathrm{conj}^{-1}(\mathrm{Im}(\lambda)),\quad
		J_2:= \ker(\rho)
	\]
	that we can construct. They are implicit in~\cite[Proof of Theorem A]{soluble}.

	\begin{proposition}\label{prop:ideals}
	    Both of $J_1$ and $J_2$ are left ideals of $A$.
	\end{proposition}

	\begin{proof}
		Clearly $J_1$ is a~subgroup of $(A,\cdot)$. Let $j\in J_1$ and write $\mathrm{conj}(j) = \lambda_b$ for some $b\in A$. For any $a\in A$, it then follows that
		\begin{align*}
			\mathrm{conj}(\lambda_a (j)) = \lambda_a\mathrm{conj}(j)\lambda_a^{-1} = \lambda_{a}\lambda_b\lambda_a^{-1}\in \mathrm{Im}(\lambda),
		\end{align*}
		which yields $\lambda_a (j)\in J_1$. This shows that $J_1$ is a~left ideal of $A$.

		Clearly $J_2$ is a~subgroup of $(A,\circ)$. For any $a\in A$ and $j\in J_2$, we have
		\[
			j \circ a~= \rho_j (a)\cdot j = a~\cdot j,
		\]
		and so $J_2$ is also a~subgroup of $(A,\cdot)$. The above further implies that
		\[
			a\circ j \circ\overline{a} = a\circ (\overline{a}\cdot j) = (a\circ \overline{a})\cdot a^{-1}\cdot (a\circ j) = \lambda_a (j).
		\]
		Since $J_2$ is in fact a~normal subgroup of $(A,\circ)$, we obtain $\lambda_a (j)\in J_2$. This proves that $J_2$ is a~left ideal of $A$.
	\end{proof}

	\section{Proof of Theorem~\ref{thm2}}

	In this section, let $A=(A,\cdot,\circ)$ be a~finite skew brace with $(A,\cdot)\simeq T^n$, where $T$ is a~non-abelian simple group and $n\in\mathbb{N}$. Also, let
	\[
		\Gamma:=\mathrm{Im}(\lambda)\mathrm{Im}(\rho),\quad J_1:=\mathrm{conj}^{-1}(\mathrm{Im}(\lambda)),\quad J_2:=\ker(\rho)
	\]
	be defined as in the previous section. By Proposition~\ref{prop:trifactor}, we have
	\begin{equation}\label{eqn:Gamma}
	    \mathrm{Inn}(A,\cdot )\leq \Gamma\leq \mathrm{Aut}(A,\cdot).
	\end{equation}
	By Proposition~\ref{prop:ideals}, we know that $J_1$ and $J_2$ are left ideals of $A$.

	\subsection{The case when $n=1$}

	First, suppose that $A$ is almost trivial. Then the left ideals of $A$ are exactly the normal subgroups of $(A,\cdot)$. But $(A,\cdot)$ is non-abelian simple, so clearly $A$ is left-simple.

	Conversely, suppose that $A$ is left-simple. Then $A$ is non-trivial, namely, $a\circ b\neq a\cdot b$ for some $a,b\in A$, for otherwise any subgroup of $(A,\cdot)$ would be a~left ideal of $A$. We suppose for contradiction that $A$ is not almost trivial. Since $A$ is neither trivial nor almost trivial, we obtain from~\cite[Theorem 1.1]{simple} that $(A,\circ)$ is not simple (the result of~\cite{simple} is stated in terms of Hopf--Galois structures, but using~\cite[Theorem 3.1]{ST}, for example, to translate it into the language of skew braces, it says that a~finite skew brace with a~non-abelian simple multiplicative group is either trivial or almost trivial). We then have $(A,\circ)\not\simeq (A,\cdot)$. By order consideration, it follows that
	\begin{equation}\label{eqn:rule out}
	    \mathrm{Im}(\lambda)\not\supseteq \mathrm{Inn}(A,\cdot)
		\quad \mbox{and} \quad
		\mathrm{Im}(\rho)\not\supseteq \mathrm{Inn}(A,\cdot).
	\end{equation}
	Indeed, note that $\mathrm{Inn}(A,\cdot)\simeq (A,\cdot)$ has order $|A|$, while
	\[
		\mathrm{Im}(\lambda) \simeq (A,\circ)/\ker(\lambda)
		\quad \mbox{and} \quad
		\mathrm{Im}(\rho)\simeq (A,\circ)/\ker(\rho)
	\]
	have order at most $|A|$. This implies that
	\[
		\begin{cases}
			(A,\circ)\simeq \mathrm{Im}(\lambda) = \mathrm{Inn}(A,\cdot)\simeq (A,\cdot)&\mbox{if }\mathrm{Im}(\lambda) \supseteq\mathrm{Inn}(A,\cdot),\\
			(A,\circ)\simeq \mathrm{Im}(\rho) = \mathrm{Inn}(A,\cdot)\simeq (A,\cdot)&\mbox{if }\mathrm{Im}(\rho) \supseteq\mathrm{Inn}(A,\cdot),
		\end{cases}
	\]
	and we have a~contradiction in both cases.

	Observe that $J_1\neq A$ because $\mathrm{Im}(\lambda)\not\supseteq\mathrm{Inn}(A,\cdot)$, and $J_2\neq A$ because $A$ is not almost trivial. Since $J_1,J_2$ are left ideals of $A$ by Proposition~\ref{prop:ideals}, the left-simplicity of $A$ implies that $J_1 =J_2 =1$.

	Note that $\Gamma$ is an almost simple group with socle $(A,\cdot)\simeq T$ by~\eqref{eqn:Gamma}. We consider its factorization $\Gamma = \mathrm{Im}(\lambda)\mathrm{Im}(\rho)$. By~\eqref{eqn:rule out}, both of the factors are core-free, in the sense that they do not contain the socle. Moreover:
	\begin{itemize}
		\item Since $J_1 =1$, the factor $\mathrm{Im}(\lambda)$ intersects trivially with $\mathrm{Inn}(A,\cdot)$ and so it embeds into $\mathrm{Out}(A,\cdot)$. Since $\mathrm{Out}(A,\cdot)$ is solvable by Schreier Conjecture (a consequence of the Classification of Finite Simple Groups), we deduce that $\mathrm{Im}(\lambda)$ is also solvable. Also, $\mathrm{Im}(\lambda)\neq 1$ because $A$ is non-trivial.

		\item Since $J_2 =1$, the factor $\mathrm{Im}(\rho)$ is isomorphic to $(A,\circ)$.
	\end{itemize}
	Fortunately, the core-free factorizations of finite almost simple groups with a~solvable factor have been classified in~\cite[Theorem 1.1]{AS}. Put $H:= \mathrm{Im}(\lambda)$, which is solvable, and $K:= \mathrm{Im}(\rho)$. Identifying $(A,\cdot)$ with $T$, we have:
	\begin{enumerate}[(1)]
		\item $|H|\neq 1$,\, $|H\cap \mathrm{Inn}(T)|=1$;

		\item $|H|$ divides $|\mathrm{Out}(T)|$;

		\item $|K|=|T|$.
	\end{enumerate}
	From Proposition~\ref{prop:trifactor}, we also know that
	\[
		\frac{K}{K\cap \mathrm{Inn}(T)} \simeq \frac{K\mathrm{Inn}(T)}{\mathrm{Inn}(T)} =\frac{H\mathrm{Inn}(T)}{\mathrm{Inn}(T)} \simeq \frac{H}{H\cap \mathrm{Inn}(T)}.
	\]
	Together with (1) and (3), this yields the equality:
	\begin{enumerate}[(1)] \setcounter{enumi}{+3}
		\item $|T| = |H| |K\cap \mathrm{Inn}(T)|$.
	\end{enumerate}
	We now use~\cite{AS} to show that all candidates of $(H,K)$ are impossible.

	\begin{remark}
		In the following tables, the orders of certain groups are listed. For the linear groups, they can be constructed in \textsc{Magma} via
		\begin{itemize}
			\item \texttt{PSL(n,q)} for the projective special linear group on $\mathrm{GF}(q)^n$;

			\item \texttt{PSU(n,q)} for the projective special unitary group on $\mathrm{GF}(q^2 )^n$;

			\item \texttt{PSp(2n,q)} for the projective sympletic group on $\mathrm{GF}(q)^{2n}$;

			\item\texttt{PSO(2n+1,q)} for the projective special orthogonal group on $\mathrm{GF}(q)^{2n+1}$ (the symbol $\mathrm{P}\Omega_{2n+1}(q)$ denotes its derived subgroup);

			\item\texttt{PSOPlus(2n,q)} and \texttt{PSOMinus(2n,q)}, respectively, for the projective special orthogonal groups on $\mathrm{GF}(q)^{2n}$ with Witt indices $n$ and $n-1$ (the symbols $\mathrm{P}\Omega_{2n}^+ (q)$ and $\mathrm{P}\Omega_{2n}^- (q)$ denote their derived subgroups);

			\item \texttt{AGammaL(n,q)} for the affine semilinear group on $\mathrm{GF}(q)^n$;

			\item \texttt{ChevalleyGroup("G",2,q)} for the Chevalley group or exceptional group of type $G_2$ over the field $\mathrm{GF}(q)$.
		\end{itemize}
		For the non-projective linear groups, simply remove \texttt{P} in the above. For the Mathieu groups, they can be constructed in \texttt{GAP} via
		\begin{itemize}
			\item \texttt{MathieuGroup(n)} for the Mathieu group acting on $n$ objects.
		\end{itemize}
		The orders of these groups may then be computed. The orders of the outer automorphism groups $\mathrm{Out}(T)$ may be computed via
		\[
			|\mathrm{Out}(T)| = |\mathrm{Aut}(T)|/|\mathrm{Inn}(T)| = |\mathrm{Aut}(T)|/|T|.
		\]
		For the symmetric and alternating groups, their orders are obvious, so we do not list them explicitly.
	\end{remark}

	Suppose first that $K$ is also solvable. By~\cite[Proposition 4.1]{AS}, for
	\[
		(H',K') = (H,K)\quad\mbox{or}\quad (H',K') = (K,H),
	\]
	the next two tables give all the possibilities.
    \begin{table}[!ht]
        \centering
    	\begin{tabular}{|C{2cm}|C{5cm}|C{4.25cm}|}
    		\hline
    		$T$ & $|K'\cap \mathrm{Inn}(T)|$ is divisible by & $|K'\cap \mathrm{Inn}(T)|$ divides\\\hline\hline
    		$\mathrm{PSL}_2 (q)$ & $q$ & $\frac{q(q-1)}{(2,q-1)}$\\
    		\hline
    	\end{tabular}
    \end{table}
    
    \begin{longtable}{|C{2cm}|C{1.75cm}|C{3.5cm}|C{3.5cm}|}
		\hline
		$T$ & $|\mathrm{Out}(T)|$ & $|H'|$ is divisible by & $|K'|$ is divisible by\\\hline\hline
		$\mathrm{PSL}_2 (7)$ & $2$ & $7 $ & $|\mathrm{S}_4 |$\\\hline
		$\mathrm{PSL}_2 (11)$ & $2$ & $11\cdot 5 $ & $|\mathrm{A}_4 |$\\\hline
		$\mathrm{PSL}_2 (23)$& $2$ & $23\cdot 11 $& $|\mathrm{S}_4 |$\\\hline
		$\mathrm{PSL}_3 (3)$ & $2$ & $13 $ & $3^2\cdot 2\cdot |\mathrm{S}_4 |$\\\hline
		$\mathrm{PSL}_3 (3)$ & $2$ &$13\cdot 3 $& $|\mathrm{A}\Gamma\mathrm{L}_1 (9)|$\\\hline
		$\mathrm{PSL}_3 (4)$ &$12$& $7\cdot 3\cdot |\mathrm{S}_3 | $ & $2^4\cdot 3\cdot |\mathrm{D}_{10}|\cdot 2$\\\hline
		$\mathrm{PSL}_3 (8)$ & $6$ & $73\cdot 9 $ & $2^{3+6}\cdot 7^2\cdot 3 $\\\hline
		$\mathrm{PSU}_3 (8)$ & $18$ & $57\cdot 9 $& $2^{3+6}\cdot 63\cdot 3 $\\\hline
		$\mathrm{PSU}_4 (2)$& $2$ &$2^4\cdot 5$ & $3^{1+2}\cdot 2 \cdot |\mathrm{A}_4 | $\\\hline
		$\mathrm{PSU}_4 (2)$& $2$ & $2^4\cdot |\mathrm{D}_{10}|$& $3^{1+2}\cdot 2\cdot |\mathrm{A}_4 | $\\\hline
		$\mathrm{PSU}_4 (2)$& $2$&$2^4\cdot 5\cdot 4$& $3^{3}\cdot |\mathrm{S}_3 | $\\\hline
		$\mathrm{M}_{11}$ & $1$ & $11\cdot 5$ & $|\mathrm{M}_9 |\cdot 2$\\
		\hline
	\end{longtable}
    
	In the second table, observe that
	\[
		|\mathrm{A}\Gamma\mathrm{L}_1 (9)| = 144, \quad |\mathrm{D}_{10}|=10,
		\quad |\mathrm{M}_9 |=72.
	\]
	We see that neither $|H'|$ nor $|K'|$ divides $|\mathrm{Out}(T)|$, which contradicts (2).

	In the first table, we must have $(H',K') = (H,K)$ by (1). Note that
	\begin{align*}
		|T| &= |\mathrm{PSL}_2 (q)| = \frac{q(q-1)(q+1)}{(2,q-1)},\\ |\mathrm{Out}(T)| &=|\mathrm{Out}(\mathrm{PSL}_2 (q))|= (2,q-1)f,
	\end{align*}
	where $q =p^f$ and $p$ is a~prime. We then see from (2) and the third column of the table that $|H||K\cap \mathrm{Inn}(T)|$ has to divide
	\[
		(2,q-1)f \cdot \frac{q(q-1)}{(2,q-1)} = q(q-1)f,
	\]
	which is strictly smaller than $|T|$ because
	\[
		f < \frac{2^f +1}{2}\leq \frac{q+1}{(2,q-1)}.
	\]
	This yields a~contradiction to (4).

	We have shown that $K$ must be insolvable. From here, we argue depending on the type of $T$. Let us remark that $T$ cannot be an exceptional group of Lie type by~\cite[Proposition~4.2]{AS}.

	\fbox{I. $T$ is an alternating group}

	The possibilities for $(H,K)$ are given in~\cite[Proposition 4.3 (a)$\sim$(f)]{AS}. Let $T = A_m$, where $m\geq 5$, and put $\Omega_m = \{1,2,\dots,m\}$. For (a), it says $H$ is transitive on $\Omega_m$, and so $|H| \geq m$. For (b), it says $H$ is $2$-homogeneous on $\Omega_m$, namely, $H$ is transitive on the $2$-subsets of $\Omega_m$, and so $|H| \geq \frac{m(m-1)}{2}$. For (c) and (f), there is a~list of candidates for $H$, and we see that $|H|\geq 5$. Since $|\mathrm{Out}(T)|\in \{2,4\}$, we obtain a~contradiction to (2) in these cases. For both (d) and (e), it states $T = A_6$ and $K\in \{\mathrm{PSL}_2 (5),\mathrm{PGL}_2 (5)\}$. But then this yields $|T| = 360$ while $|K|\in \{60,120\}$, and this contradicts (3).

	\fbox{II. $T$ is a~sporadic group}

	The possibilities for $(H,K)$ are given in~\cite[Proposition 4.4 (a)$\sim$(c)]{AS}. For (a) and (b), it says $(T,K) =( \mathrm{M}_{12},\mathrm{M}_{11}),  (\mathrm{M}_{24},\mathrm{M}_{23})$, respectively, which contradicts (3). For (c), a~list of candidates for $H$ is given, and we find that $|H|\geq 3$. Since $|\mathrm{Out}(T)|\in \{1,2\}$, this contradicts (2).

	\fbox{III. $T$ is a~classical group of Lie type and is not an alternating group}

	The possibilities for $(H,K)$ are given in~\cite[Tables 1.1 and 1.2]{AS}. However, they only give an upper bound for $H\cap \mathrm{Inn}(T)$, which is not enough for us. For~\cite[Table 1.1]{AS}, we shall further apply~\cite[Theorem 1]{factor}, which provides a~lower bound for $|H|$ and we can derive a~contradiction to (2). For~\cite[Table 1.2]{AS}, we shall use the description of $K\cap\mathrm{Inn}(T)$ and obtain a~contradiction to (4). We summarize all the possibilities in the two tables below.
	\begin{longtable}{|C{0.45cm}||C{5.25cm}|C{5.5cm}|}
		\hline
		&$T$ & $|H|$ is divisible by\\\hline\hline
		(a) &$\mathrm{PSL}_n (q)$, $n\geq 2$ & $\frac{q^n -1}{q-1}$ \\\hline
		(b) & $\mathrm{PSL}_4 (q)$ & $\frac{q^3 (q^3 -1)}{(2,q-1)}$\\
		\hline
		(c) &$\mathrm{PSp}_{2m}(q)$, $m\geq 2$, $q$ even & $q^m (q^m -1)$ \\\hline
		(d) & $\mathrm{PSp}_4 (q)$, $q$ even & $q^2 (q^2 -1)$\\\hline
		(e) & $\mathrm{PSp}_4 (q)$, $q$ odd & $q^3 (q^2 -1)$\\
		\hline
		(f) &$\mathrm{PSU}_{2m}(q)$, $m\geq 2$ & $\frac{q^{2m}(q^{2m}-1)}{q+1}$ \\\hline
		(g) &$\Omega_{2m+1}(q)$, $m\geq 3$, $q$ odd & $\frac{1}{2}q^{m(m+1)/2}(q^m -1)$ \\\hline
		(h) &$\mathrm{P}\Omega^+_{2m}(q)$, $m\geq 5$ & $\frac{q^m (q^m -1)}{(2,q-1)}$ \\\hline
		(i) & $\mathrm{P}\Omega_8^+ (q)$ & $\frac{q^4 (q^4 -1)}{(2,q-1)}$\\\hline
	\end{longtable}

	\begin{longtable}{|C{2cm}|C{3.5cm}|C{1.5cm}|C{4cm}|}
		\hline
		$T$ & $|T|$ &$|\mathrm{Out}(T)|$ & $|K\cap \mathrm{Inn}(T)|$ divides \\\hline\hline
		$\mathrm{PSL}_2 (11)$ & $660$ & $2$ &$|\mathrm{A}_5 |$ \\\hline
		$\mathrm{PSL}_2 (16)$ & $4080$ & $4$ &$|\mathrm{A}_5 |$ \\\hline
		$\mathrm{PSL}_2 (19)$ & $3420$ &$2$ &$|\mathrm{A}_5 |$ \\\hline
		$\mathrm{PSL}_2 (29)$ & $12180$ & $2$ &$|\mathrm{A}_5 |$ \\\hline
		$\mathrm{PSL}_2 (59)$ & $102660$& $2$ &$|\mathrm{A}_5 |$ \\\hline
		$\mathrm{PSL}_4 (3)$ & $6065280$ &$4$&$3^3\cdot |\mathrm{PSL}_{3}(3)|$\\\hline
		$\mathrm{PSL}_4 (3)$ & $6065280$&$4$& $4\cdot |\mathrm{PSL}_2 (9)|\cdot 2$ \\\hline
		$\mathrm{PSL}_4 (4)$ & $987033600$ &$4$& $5\cdot |\mathrm{PSL}_2 (16)|\cdot 2$ \\\hline
		$\mathrm{PSL}_5 (2)$ & $9999360$&$2$& $2^6\cdot |\mathrm{S}_3 |\cdot |\mathrm{PSL}_3 (2)|$\\\hline
		$\mathrm{PSp}_4 (3)$ & $25920$&$2$& $2^4\cdot |\mathrm{A}_5 |$\\\hline
		$\mathrm{PSp}_4 (3)$&$25920$&$2$& $2^4\cdot |\mathrm{S}_6 |$ \\\hline
		$\mathrm{PSp}_4 (5)$& $4680000$&$2$& $|\mathrm{PSL}_{2}(5^2 )|\cdot 2$ \\\hline
		$\mathrm{PSp}_4 (7)$&$138297600$ &$2$& $|\mathrm{PSL}_{2}(7^2 )|\cdot 2$ \\\hline
		$\mathrm{PSp}_4 (11)$&$12860654400$ &$2$& $|\mathrm{PSL}_{2}(11^2 )|\cdot 2$ \\\hline
		$\mathrm{PSp}_4 (23)$&$20674026236160$& $2$& $|\mathrm{PSL}_{2}(23^2 )|\cdot 2$ \\\hline
		$\mathrm{Sp}_6 (2)$&$1451520$&$1$&$|\mathrm{S}_8 |$ \\\hline
		$\mathrm{PSp}_6 (3)$&$4585351680$&$2$&$|\mathrm{PSL}_{2}(27)|\cdot 3$\\
		\hline
		$\mathrm{PSU}_3 (3)$&$6048$&$2$&$|\mathrm{PSL}_{2}(7)|$\\\hline
		$\mathrm{PSU}_3 (5)$&$126000$&$6$&$|\mathrm{A}_7 |$\\\hline
		$\mathrm{PSU}_4 (3)$&$3265920$&$8$&$|\mathrm{PSL}_3 (4)|$\\\hline
		$\mathrm{PSU}_4 (8)$&$34693789777920$&$6$&$2^{12}\cdot|\mathrm{SL}_2 (64)|\cdot 7$\\\hline
		$\Omega_7 (3)$&$4585351680$&$2$&$|\mathrm{G}_2 (3)|$\\\hline
		$\Omega_7 (3)$&$4585351680$&$2$&$|\mathrm{Sp}_6 (2)|$\\\hline
		$\Omega_9 (3)$&$65784756654489600$&$2$&$|\Omega_8^- (3)|\cdot 2$\\\hline
		$\Omega_8^+ (2)$&$174182400$&$6$&$|\mathrm{Sp}_6 (2)|$\\\hline
		$\Omega_8^+ (2)$&$174182400$&$6$&$|\mathrm{A}_9 |$\\\hline
		$\mathrm{P}\Omega_8^+ (3)$&$4952179814400$&$24$&$|\Omega_7 (3)|$\\\hline
		$\mathrm{P}\Omega_8^+ (3)$&$4952179814400$&$24$&$|\Omega_8^+ (2)|$\\\hline
	\end{longtable}
	In the second table, let $\mathcal{K}$ be the number in the last column. We deduce from (2) that $|H||K\cap \mathrm{Inn}(T)|$ has to divide $|\mathrm{Out}(T)|\mathcal{K}$. Note that:
	\begin{longtable}{|C{2cm}|C{3.45cm}||C{2cm}|C{3.45cm}|}
		\hline
		$|\mathrm{PSL}_{3}(3)|$&$5616$ &$|\mathrm{PSL}_{2}(7)|$&$168$\\\hline$|\mathrm{PSL}_2 (9)|$ & $360$ &$|\mathrm{PSL}_3 (4)|$&$20160$\\\hline
		$|\mathrm{PSL}_2 (16)|$ & $4080$ &$|\mathrm{SL}_2 (64)|$&$262080$\\\hline$|\mathrm{PSL}_3 (2)|$&$168$&$|\mathrm{G}_2 (3)|$&$4245696$\\\hline
		$|\mathrm{PSL}_{2}(5^2 )|$&$7800$ &$|\mathrm{Sp}_6 (2)|$&$1451520$\\\hline$|\mathrm{PSL}_{2}(7^2 )|$&$58800$&$|\Omega_8^- (3)|$&$10151968619520$\\\hline
		$|\mathrm{PSL}_{2}(11^2 )|$&$885720$&
		$|\Omega_7 (3)|$&$4585351680$\\\hline$|\mathrm{PSL}_{2}(23^2 )|$&$74017680$&$|\Omega_8^+ (2)|$&$174182400$\\\hline
		$|\mathrm{PSL}_{2}(27)|$&$9828$&&\\
		\hline
	\end{longtable}
	It is then straightforward to check that
	\[
		|\mathrm{Out}(T)|\mathcal{K} < |T|
	\]
	in every row of the second table. This yields a~contradiction to (4).

	In the first table, write $q=p^f$, where $p$ is a~prime. As listed in~\cite[Table 2.1]{AS}, for example, we know that:
	\begin{longtable}{|C{6.5cm}|C{4.5cm}|}
		\hline
		$T$ & $|\mathrm{Out}(T)|$\\
		\hline\hline
		$\mathrm{PSL}_2 (q)$ & $(2,q-1)\cdot f$\\\hline
		$\mathrm{PSL}_n (q)$, $n\geq 3$ & $(n,q-1)\cdot 2\cdot f$\\\hline
		$\mathrm{PSU}_n (q)$, $n\geq 3$ & $(n,q+1)\cdot 2f$\\\hline
		$\mathrm{PSp}_{2m}(q)$, $(m,p)\neq (2,2)$ & $(2,q-1)\cdot f$\\\hline
		$\mathrm{PSp}_4 (q)$, $q$ even & $2f$\\\hline
		$\Omega_{2m+1}(q)$, $m\geq 3$, $q$ odd & $2\cdot f$\\\hline
		$\mathrm{P}\Omega_{2m}^+ (q)$, $m\geq 5$, $q^m\not\equiv 1\pmod{4}$ & $2\cdot (2,q-1)\cdot f$\\\hline
		$\mathrm{P}\Omega_{2m}^+ (q)$, $m\geq 5$, $q^m\equiv 1\pmod{4}$ & $8\cdot f$\\\hline
		$\mathrm{P}\Omega_{8}^+ (q)$& $ (2+(2,q-1))! \cdot f$\\\hline
	\end{longtable}
	For an integer $z$, write $v_p (z)$ for the $p$-adic valuation of $z$. Using the trivial bound that $v_p (2)\leq 1$ when $2$ divides $|\mathrm{Out}(T)|$, we see that:
	\begin{longtable}{|C{0.45cm}||C{5.25cm}|C{5.5cm}|}
		\hline
		&$v_p (|\mathrm{Out}(T)|)$ is at most & $v_p (|H|)$ is at least\\\hline\hline
		(b) & $1+v_p (f)$ & $3f$\\
		\hline
		(c) &$1+v_p (f)$, $p=2$ & $mf$, $m\geq 2$ \\\hline
		(d) & $1+v_p (f)$, $p=2$ & $2f$\\\hline
		(e) & $v_p (f)$, $p\geq 3$ & $3f$\\
		\hline
		(f) &$1+v_p (f)$ & $2mf$, $m\geq 2$\\\hline
		(g) &$v_p (f)$, $p\geq 3$ & $\frac{m(m+1)}{2}f$, $m\geq 3$ \\\hline
		(h) &$3+v_p (f)$ & $mf$, $m\geq 5$ \\\hline
		(i) & $3+v_p (f)$ & $4f$\\\hline
	\end{longtable}
	Since $v_p (f) \leq \log_p (f) <f$ and $1\leq f$, in cases (b)$\sim$(i), we deduce that
	\[
		v_p (|\mathrm{Out}(T)|) < v_p (|H|).
	\]
	But then $|H|$ cannot divide $|\mathrm{Out}(T)|$, and this contradicts (2).

	Finally, we deal with case (a). If $n=2$, then
	\[
		|\mathrm{Out}(T)| \leq 2f < 2^f +1 \leq p^f +1 = \frac{q^2 -1}{q-1} \leq |H|,
	\]
	which contradicts (2). If $n\geq 3$ with $(fn,p) = (6,2)$, then
	\[
		\begin{cases}
			|\mathrm{Out}(T)| = 2 < \frac{2^6 -1}{2-1} \leq |H|&\mbox{for }(f,n)=(1,6),\, q=2,\\
			|\mathrm{Out}(T)| = 12 < \frac{4^3 -1}{4-1} \leq |H| & \mbox{for }(f,n)=(2,3),\, q=4,
		\end{cases}
	\]
	which again contradicts (2). If $n\geq 3$ with $(fn,p)\neq (6,2)$, then we deduce from Zsigmondy's theorem~\cite{Zsigmondy} (also see~\cite[Theorem 3]{Roitman}) that $p^{fn}-1$ has a~primitive prime divisor $\ell$. In other words, the prime $\ell$ is such that
	\begin{equation}\label{eqn:divisible}
		\ell \mid p^{fn}-1,\quad \ell\nmid p^k -1\mbox{ for all }1\leq k\leq fn-1.
	\end{equation}
	It follows that $\ell$ divides $\frac{q^n -1}{q-1}$ and hence $|H|$. By (2), we see that $\ell$ divides
	\[
		|\mathrm{Out}(T)|=(n,q-1)\cdot 2\cdot f.
	\]
	Since $\ell\nmid q-1$ and clearly $\ell\neq 2$, this means that $\ell\mid f$. But $p\equiv p^\ell \pmod{\ell}$ so we deduce from the first condition in~\eqref{eqn:divisible} that
	\[
		p^{\frac{f}{\ell}n}-1 \equiv p^{fn}-1\equiv 0\pmod{\ell}.
	\]
	This contradicts the second condition in~\eqref{eqn:divisible}.

	We have now ruled out every possibility. Therefore, the assumption that $A$ is not almost trivial is false. This completes the proof of Theorem~\ref{thm2}(a).

	\subsection{The case when $n\geq 2$.}

	Suppose that $A$ is left-simple. Note that
	\begin{equation}\label{eqn:lambdarho}
		\lambda_a\equiv \rho_a\pmod{\mathrm{Inn}(A,\cdot)}
	\end{equation}
	for all $a\in A$ by~\eqref{eqn:lambda}, and that
	\begin{equation}\label{eqn:lambda Inn}
	    \mathrm{Im}(\lambda)\not\subseteq \mathrm{Inn}(A,\cdot),
	\end{equation}
	for otherwise $T\times 1\times \cdots\times 1$ would be a~non-trivial proper left ideal of $A$. We then see that $J_1\neq A$ by order consideration. Indeed, if $J_1 =A$, namely the inclusion $\mathrm{Inn}(A,\cdot)\subseteq \mathrm{Im}(\lambda)$ holds, then we must in fact have an equality because $\mathrm{Inn}(A,\cdot)\simeq (A,\cdot)$ has order $|A|$, while
	\[
		\mathrm{Im}(\lambda) \simeq (A,\circ)/\ker(\lambda)
	\]
	has order at most $|A|$. But this contradicts~\eqref{eqn:lambda Inn}. From~\eqref{eqn:lambdarho} and~\eqref{eqn:lambda Inn}, we also obtain $J_2\neq A$. Since $J_1,J_2$ are left ideals of $A$ by Proposition~\ref{prop:ideals}, the left-simplicity of $A$ implies that $J_1 =J_2 =1$.

	Since $J_1 =1$, the image $\mathrm{Im}(\lambda)$ of $\lambda$ intersects trivially with $\mathrm{Inn}(A,\cdot)$, as claimed in the theorem.

	Since $J_2 =1$, the multiplicative group $(A,\circ)$ embeds into
	\[
		\mathrm{Aut}(A,\cdot)\simeq \mathrm{Aut}(T)^n\rtimes S_n
	\]
	via $\rho$, and let $P$ denote the projection of $\mathrm{Im}(\rho)$ onto $S_n$. Since
	\[
		\mathrm{Inn}(A,\cdot) \simeq \mathrm{Inn}(T)^n \subseteq \mathrm{Aut}(T)^n,
	\]
	the projection of $\mathrm{Im}(\lambda)$ onto $S_n$ is also equal to $P$ by~\eqref{eqn:lambdarho}. For any orbit $O$ of the natural action of $P$ on $\{1,2,\dots,n\}$, the subgroup of $(A,\cdot)$ given by
	\[
		\prod_{i\in O}T_i,
	\]
	where $T_i$ denotes the $i$th copy of $T$ in $(A,\cdot)\simeq T^n$, is then a~left ideal of $A$. But $A$ is left-simple, so necessarily $O=\{1,2,\dots,n\}$. This proves that $P$ is a~transitive subgroup of $S_n$.

	This completes the proof of Theorem~\ref{thm2}(b).

	\begin{remark}
		The condition that $(A,\circ)$ is isomorphic to a~subgroup, $X$ say, of $\mathrm{Aut}(T^n )$ whose projection onto $S_n$ is transitive is very restrictive. To see why, consider the natural homomorphism
		\[
			\phi: X \lhook\joinrel\xrightarrow{\quad} \mathrm{Aut}(T)^n\rtimes S_n\longrightarrow \mathrm{Out}(T)^n\rtimes S_n.
		\]
		Its kernel $\ker(\phi)$ lies in $\mathrm{Inn}(T)^n \simeq T^n$, and its image $\mathrm{Im}(\phi) \simeq X/\ker(\phi)$ has the same projection onto $S_n$ as $X$, which is transitive. It follows that
		\begin{equation}
			\exists R\leq T^n\mbox{ s.t. }n \mid [T^n:R]\mbox{\,\ and\,\ }[T^n:R]\mid |\mathrm{Out}(T)|^n\cdot n!.
		\end{equation}
		In particular, the prime factors of $n$ must all divide $|T|$.

		For example, take $T = (T,\cdot) = (A_5,\cdot)$. Note that
		\[
			|A_5 | = 60 = 2^2 \cdot 3 \cdot 5\quad\mbox{and}\quad|\mathrm{Out}(A_5 )|=2.
		\]
		We see that $(A_5^n,\cdot,\circ)$ can be left-simple only when $n=2^{x}3^y 5^z$, where $x,y,z\in \mathbb{N}_{\geq 0}$. Moreover, a~skew brace $(A_5^2,\cdot,\circ)$ cannot be left-simple because $A_5^2$ has no subgroup of index $2,4$, or $8$.
	\end{remark}

	\section{Connection with minimal Hopf--Galois structures}\label{sec:minimal}

	Let $L/K$ be a~finite extension of fields. A \textit{Hopf--Galois structure} on $L/K$ is a~finite-dimensional cocommutative Hopf $K$-algebra $\mathcal{H}$ that acts on $L$ via a~$K$-algebra homomorphism $\mathcal{H}\longrightarrow\mathrm{End}_K (L)$ such that the action respects the counit $\epsilon$, multiplication $m$, and comultiplication $\Delta$, in the sense that
	\[
		h(1) = \epsilon(h)(1),\quad h(xy) = m(\Delta(h)(x\otimes y))
	\]
	for all $h\in \mathcal{H},\, x,y\in L$, and the natural $K$-homomorphism
	\[
		L\otimes_K\mathcal{H}\longrightarrow \mathrm{End}_K (L)
	\]
	is bijective. The notion of a~Hopf--Galois structure was introduced by~\cite{CS} (or see~\cite{Childs-book} for detailed overview). For each Hopf $K$-subalgebra $\mathcal{H}'$ of $\mathcal{H}$, let
	\[
		L^{\mathcal{H}'} = \{x\in L: h'(x) = \epsilon(h')x \mbox{ for all }h'\in\mathcal{H}'\}
	\]
	denote its fixed field. This yields a~Hopf--Galois correspondence
	\begin{align*}
		\Phi_{\mathcal{H}}: \{\mbox{Hopf $K$-subalgebras of $\mathcal{H}$}\}& \longrightarrow \{\mbox{intermediate fields of $L/K$}\};\\
		\mathcal{H'}&\longmapsto L^{\mathcal{H}'}
	\end{align*}
	that is analogous to the Galois correspondence in Galois theory. Indeed, in the case that $L/K$ is a~$G$-Galois extension, the group ring $K[G]$ is naturally a~Hopf--Galois structure on $L/K$, and the map $\Phi_{K[G]}$ above reduces to the usual Galois correspondence
	\begin{align*}
		\{\mbox{subgroups of $G$}\}& \longrightarrow \{\mbox{intermediate fields of $L/K$}\};\\
		G'&\longmapsto L^{G'}
	\end{align*}
	because the Hopf $K$-subalgebras of $K[G]$ are exactly the group rings $K[G']$, where $G'$ ranges over the subgroups of $G$.

	It is known by~\cite{CS} that the map $\Phi_{\mathcal{H}}$ is always injective, but it need not be surjective in general. We can of course ask when $\Phi_{\mathcal{H}}$ is surjective. Here, we consider the other extreme and ask when $\Phi_{\mathcal{H}}$ is as far from being surjective as possible, i.e. when $\mathrm{Im}(\Phi_{\mathcal{H}})$ only contains $L$ and $K$. Since $\Phi_{\mathcal{H}}$ is injective, this is precisely the case when $K$ and $\mathcal{H}$ are the only Hopf $K$-subalgebras of $\mathcal{H}$. The next definition is due to~\cite{minimal}.

	\begin{definition}
		A Hopf--Galois structure $\mathcal{H}$ on $L/K$ is said to be \textit{minimal} if $\dim_K (\mathcal{H})\neq 1$ and $\mathcal{H}$ has no Hopf $K$-subalgebras other than $K$ and $\mathcal{H}$.
	\end{definition}

	The theory of Hopf--Galois structures applies to all extensions. But when restricted to separable extensions, we have a~very nice group-theoretic classification of Hopf--Galois structures thanks to~\cite{GP}. In addition, when further restricted to Galois extensions, the classification may be translated into the language of skew braces. The next theorem is from~\cite[Theorem 3.1]{ST}.

	\begin{theorem}
		Let $L/K$ be any finite Galois extension of fields with Galois group $(G,\circ)$. There is a~bijective correspondence between$:$
		\begin{enumerate}[$(1)$]
			\item the Hopf--Galois structures $\mathcal{H}$ on $L/K;$
			\item the binary operations $\cdot$ on $G$ such that $(G,\cdot,\circ)$ is a~skew brace.
		\end{enumerate}
		Specifically, such an operation $\cdot$ is associated to the Hopf--Galois structure
		\[
			L[(G,\cdot)]^{(G,\circ)} = \left\{ \sum_{\tau\in G}\ell_\tau\tau \in L[(G,\cdot)]:
			\sigma(\ell_{\tau}) = \ell_{\lambda_\sigma(\tau)} \mbox{ for all }\sigma,\tau\in G\right\}
		\]
		whose action on $L$ is defined by
		\[
			\left(\sum_{\tau\in G}\ell_\tau \tau \right)(x) = \sum_{\tau\in G}\ell_\tau\tau(x)
		\]
		for all $x\in L$, where $\lambda$ is the lambda map~\eqref{eqn:lambda map} of the skew brace $(G,\cdot,\circ)$.
	\end{theorem}

	Let us remark that the Hopf--Galois structure corresponding to the trivial skew brace is the group ring $K[G]$ and is called the \textit{classical structure}, while the Hopf--Galois structure corresponding to the almost trivial skew brace is called the \textit{canonical non-classical structure}.

	The next theorem is from~\cite[Corollary 4.1]{ST}.

	\begin{theorem}
		Let $L/K$ be any finite Galois extension of fields with Galois group $(G,\circ)$. Let $\cdot$ be any binary operation on $G$ such that $(G,\cdot,\circ)$ is a~skew brace and let $\mathcal{H}$ be the associated Hopf--Galois structure on $L/K$. For any subgroup $I$ of $G$, the following are equivalent$:$
		\begin{enumerate}[$(1)$]
			\item the intermediate field $L^I$ lies in the image of $\Phi_{\mathcal{H}};$
			\item the subgroup $I$ is a~left ideal of $(G,\cdot,\circ)$.
		\end{enumerate}
		In particular, the following are equivalent$:$
		\begin{enumerate}[$(a)$]
			\item the Hopf--Galois structure $\mathcal{H}$ is minimal$;$
			\item the skew brace $(G,\cdot,\circ)$ is left-simple.
		\end{enumerate}
	\end{theorem}

	Therefore, in the language of Hopf--Galois theory, our Theorem~\ref{thm1}, Corollary~\ref{cor}, and Theorem~\ref{thm2} give a~partial classification of the minimal Hopf--Galois structures on finite Galois extensions of fields.

    
\end{document}